\documentclass[11pt]{amsart}

\usepackage{amssymb}
\usepackage{geometry}

\usepackage[colorlinks=true,urlcolor=blue,
citecolor=red,linkcolor=blue,linktocpage,pdfpagelabels,bookmarksnumbered,bookmarksopen]{hyperref}
\usepackage[hyperpageref]{backref}

\newtheorem{lemma}{Lemma}[section]
\newtheorem{prop}[lemma]{Proposition}
\newtheorem{teo}[lemma]{Theorem}

\theoremstyle{definition}
\newtheorem{oss}[lemma]{Remark}
\newtheorem{defi}[lemma]{Definition}
\newtheorem*{ack}{Acknowledgements}

\author[Brasco]{Lorenzo Brasco}
\author[Santambrogio]{Filippo Santambrogio}

\address[L.\ Brasco]{Dipartimento di Matematica e Informatica
\newline\indent
Universit\`a degli Studi di Ferrara,
44121 Ferrara, Italy}
\address{{\it and } Institut de Math\'ematiques de Marseille
\newline\indent
Aix-Marseille Universit\'e,
Marseille, France}
\email{lorenzo.brasco@unife.it}

\address[F.\ Santambrogio]{Laboratoire de Math\'ematiques d'Orsay 
\newline\indent 
Univ. Paris-Sud, CNRS, Universit\'e Paris-Saclay, 91405 Orsay Cedex, France}
\email{filippo.santambrogio@math.u-psud.fr}

\title[An estimate for the $p-$Laplacian]{A sharp estimate \`a la Calder\'on-Zygmund\\ for the $p-$Laplacian}

\numberwithin{equation}{section}


\begin{document}

\begin{abstract}
We consider local weak solutions of the Poisson equation for the $p-$Laplace operator. We prove a higher differentiability result, under an essentially sharp condition on the right-hand side. The result comes with a local scaling invariant a priori estimate.
\end{abstract}

\maketitle

\begin{center}
\begin{minipage}{11cm}
\small
\tableofcontents
\end{minipage}
\end{center}

\section{Introduction}

\subsection{The problem}
In this paper we are concerned with local or global $W^{1,p}$ solutions to the Poisson equation for the $p-$Laplace operator, i.e.
\begin{equation}
\label{plaplace}
-\Delta_p u:=-\mathrm{div}(|\nabla u|^{p-2}\,\nabla u)=f,\qquad \mbox{ in }\Omega,
\end{equation}
with $\Omega\subset \mathbb{R}^N$ open set.
Our analysis is confined to the super-quadratic case, i.e. throughout the whole paper we consider $p>2$.
For $f\equiv 0$, we know that 
\begin{equation}
\label{karen}
|\nabla u|^\frac{p-2}{2}\,\nabla u\in W^{1,2}_{\rm loc}(\Omega;\mathbb{R}^N), 
\end{equation}
This is a well-known regularity result which dates back to Uhlenbeck, see \cite[Lemma 3.1]{Uh}. We refer to \cite[Proposition 3.1]{Mi03} for a generalization of this result. If $f$ is smooth enough, the same result is easily seen to be still true. For example, it is sufficient to take 
\begin{equation}
\label{fsmooth}
f\in W^{1,p'}_{\rm loc}(\Omega),
\end{equation}
where $p'=p/(p-1)$.
However, it is easy to guess that assumption \eqref{fsmooth} is far from being optimal: in the limit case $p=2$, \eqref{karen} boils down to 
\[
\nabla u\in W^{1,2}_{\rm loc}(\Omega;\mathbb{R}^N).
\]
Then from {\it Calder\'on-Zygmund estimates} for the Laplacian, we know that this is true if (and only if)
\[
f\in L^2_{\rm loc}(\Omega).
\]
Thus in this case $f\in L^2_{\rm loc}$ would be the sharp assumption in order to get Uhlenbeck's result.
The main concern of this work is to prove Uhlenbeck's result for solutions of \eqref{plaplace}, under sharp assumptions on $f$.

\subsection{The main result} In this paper we prove the following result. We refer to Section \ref{sec:2} for the notation.
\begin{teo}
\label{teo:local}
Let $p>2$ and let $U\in W^{1,p}_{\rm loc}(\Omega)$ be a local weak solution of equation \eqref{plaplace}. If 
\[
f\in W^{s,p'}_{\rm loc}(\Omega)\qquad \mbox{ with } \frac{p-2}{p}<s\le 1,
\]
then
\[
\mathcal{V}:=|\nabla U|^\frac{p-2}{2}\,\nabla U\in W^{1,2}_{\rm loc}(\Omega;\mathbb{R}^N),
\] 
and
\[
\nabla U\in W^{\sigma,p}_{\rm loc}(\Omega;\mathbb{R}^N),\qquad \mbox{ for }0<\sigma<\frac{2}{p}.
\]
Moreover, for every pair of concentric balls $B_r\Subset B_R\Subset\Omega$ and every $j=1,\dots,N$ we have
the local scaling invariant estimate
\begin{equation}
\label{filipposobolev}
\begin{split}
\int_{B_r} \left|\mathcal{V}_{x_j}\right|^2\,dx&\le \frac{C}{(R-r)^2}\int_{B_R} |\nabla U|^p\,dx+C\left(R^{\left(s-\frac{p-2}{p}\right)}\,[f]_{W^{s,p'}(B_R)}\right)^{p'},
\end{split}
\end{equation}
for a constant $C=C(N,p,s)>0$ which blows-up as $s\searrow (p-2)/p$.
\end{teo}
\begin{oss}[Sharpness of the assumption]
The assumption on $f$ in the previous result is essentially sharp, in the sense that {\it the result is false for} $s<(p-2)/p$, see Section \ref{sec:5} for an example. Also observe that
\[
\frac{p-2}{p}\searrow 0 \quad \mbox{ as }\quad p\searrow 2,
\]
thus the assumption on $f$ is consistent with the case of the Laplacian recalled above.
\end{oss}
\begin{oss}[Comparison with other results]
The interplay between regularity of the right-hand side $f$ and that of the vector field $\mathcal{V}$ has been considered in detail by Mingione in \cite{Mi07}. However, our Theorem \ref{teo:local} does not superpose with the results of \cite{Mi07}. Indeed, the point of view in \cite{Mi07} is slightly different: the main concern there is to obtain (fractional) differentiability of the vector field
\[
\mathcal{V}:=|\nabla U|^\frac{p-2}{2}\,\nabla U,
\]
when $f$ is not regular. In particular, in \cite{Mi07} the right-hand side $f$ may not belong to the relevant dual Sobolev space and the concept of solution to \eqref{plaplace} has to be carefully defined.
\par
In order to give a flavour of the results in \cite{Mi07}, we recall that \cite[Theorem 1.3]{Mi07} proves that if $2<p\le N$ and $f$ is a Radon measure, then
\[
\mathcal{V}\in W^{\frac{\tau}{2},\frac{2}{p'}}_{\rm loc}(\Omega;\mathbb{R}^N),\qquad \mbox{ for every } 0<\tau<p'. 
\]
When the Radon measure $f$ is more regular, accordingly one can improve the differentibility of $\mathcal{V}$. For example, \cite[Theorem 1.6]{Mi07} proves that if $2<p\le N$ and $f\in L^{1,\lambda}(\Omega)$ with $N-p<\lambda\le N$, then\footnote{The result of \cite[Theorem 1.6]{Mi07} is not stated for $\mathcal{V}$, but rather directly for $\nabla U$. However, an inspection of the proof reveals that one has the claimed regularity of $\mathcal{V}$, see \cite[proof of Theorem 1.6, page 33]{Mi07}.}
\[
\mathcal{V}\in W^{\frac{\tau}{2},2}_{\rm loc}(\Omega;\mathbb{R}^N),\qquad \mbox{ for every } 0<\tau<p'\,\left(1-\frac{N-\lambda}{p}\right). 
\]
Here $L^{1,\lambda}(\Omega)$ is the usual {\it Morrey space}. In this case, the assumption on $f$ guarantees that a solution to \eqref{plaplace} can be defined in variational sense, we refer to \cite{Mi07} for more details.
\par
Some prior results are also due to J. Simon, who proved for example the global regularity
\[
\mathcal{V}\in W^{\frac{\tau}{2},2}(\mathbb{R}^N),\qquad \mbox{ for every } 0<\tau<p',
\]
for solutions $U$ in the whole $\mathbb{R}^N$ with right-hand side $f\in L^{p'}(\mathbb{R}^N)$, see \cite[Theorem 8.1]{SimonR5}.
Finally, even if it is concerned with the solution $U$ rather than the vector field $\mathcal{V}$, we wish to mention a result by Savar\'e contained in \cite{Savare}. This paper is concerned with {\it global regularity} for solutions of \eqref{plaplace} satisfying homogeneous Dirichlet boundary conditions.
In \cite[Theorem 2]{Savare}, it is shown that 
\[
f\in W^{-1+\frac{\tau}{p'},p'}(\Omega) \quad \Longrightarrow \quad U\in W^{1+\frac{\tau}{p},p}(\Omega),\qquad \mbox{ for every } 0<\tau<1,
\]
when $\partial\Omega$ is Lipschitz continuous. This gives a regularity gain on the solution $U$ of
\[
2+\tau\,\left(\frac{1}{p}-\frac{1}{p'}\right),
\]
orders of differentiability, compared to the right-hand side $f$.
It is interesting to notice that by formally taking $\tau=2$ in the previous implication, this essentially gives the regularity gain of Theorem \ref{teo:local}. %, once we observe that $\mathcal V\in H^1_{\rm loc}$ roughly corresponds to $\nabla U\in W^{\frac 2p,p}_{\rm loc}$, due to the composition with the H\"older continuous function $\mathbb{R}^N\ni z\mapsto |z|^{\frac 2p-1}z$.
 %  $\mathcal V$ instead of $u$, the present result should be compared with the result of \cite[Theorem 2]{Savare} by Savar\'e. Indeed, if we consider that $\mathcal V\in H^1_{loc}$ represents some sort of second derivatives for $u$, we can see that our result corresponds to a gain of almost $1+2/p$ orders of differentiability for $u$, compared to the datum $f$. This is exactly the situation of \cite[Theorem 2]{Savare}, where $f\in W^{-1+\lambda/p',p'}$ provides $u\in W^{1+\lambda/p,p}$ for every $\lambda\in [0,1[$.
\end{oss}

\subsection{About the proof}
Let us try to explain in a nutshell the key point of estimate \eqref{filipposobolev}. For the sake of simplicity, let us assume that $U$ is smooth (i.e. $U\in C^2$) and explain how to arrive at the a priori estimate \eqref{filipposobolev}. The rigourous proof is then based on a standard approximation procedure. 
\par
For ease of notation, we set
\[
G(z)=\frac{|z|^p}{p},\qquad \mbox{ for }z\in\mathbb{R}^N,
\]
then $U\in W^{1,p}_{\rm loc}(\Omega)$ verifies
\[
\int \langle \nabla G(\nabla U),\nabla \varphi\rangle\,dx=\int f\,\varphi,
\]
for every compactly supported test function $\varphi$.
Uhlenbeck's result is just based on differentiating this equation in direction $x_j$ and then testing it against\footnote{Of course, this test function is not compactly supported. Actually, to make it admissible we have to multiply it by a cut-off function, see Proposition \ref{prop:sobolevuniforme}. This introduces some lower-order terms in the estimate, which are not essential at this level and would just hide the idea of the proof.} $U_{x_j}$. This yields
\begin{equation}
\label{equazionederivata}
\int \langle D^2 G(\nabla U)\,\nabla U_{x_j},\nabla U_{x_j}\rangle\,dx=\int f_{x_j}\,U_{x_j}\,dx.
\end{equation}
By using the convexity properties of $G$, it is easy to see that
\[
\int |\mathcal{V}_{x_j}|^2\,dx=\int \left|\left(|\nabla U|^\frac{p-2}{2}\,\nabla U\right)_{x_j}\right|^2\,dx\lesssim\int \langle D^2 G(\nabla U)\,\nabla U_{x_j},\nabla U_{x_j}\rangle\,dx.
\]
The main difficulty is now to estimate the right-hand side of \eqref{equazionederivata}, without using first order derivatives of $f$. A first na\"ive idea would be to integrate by parts: of course, this can not work, since this would let appear the Hessian of $U$ on which we do not have any estimate. A more clever strategy is to {\it integrate by parts in fractional sense}, i.e. use a duality-based inequality of the form
\[
\left|\int f_{x_j}\,U_{x_j}\,dx\right|\le \|f_{x_j}\|_{W^{s-1,p'}}\,\|U_{x_j}\|_{W^{1-s,p}},
\]
where $W^{s-1,p'}$ is just the topological dual of $W^{1-s,p}$.
The main point is then to prove that
\[
``\mbox{\it taking a fractional derivative of negative order of $f_{x_j}$}
\]
\[
\mbox{\it gives a fractional derivative of positive order\,}''
\] 
i.e. we use that
\begin{equation}
\label{voto}
\|f_{x_j}\|_{W^{s-1,p'}}\lesssim \|f\|_{W^{s,p'}},
\end{equation}
see Theorem \ref{teo:filipponecas} below. 
\par
In order to conclude, we still have to control the term containing fractional derivatives of $U_{x_j}$. This can be absorbed in the left-hand side, once we notice that $U_{x_j}$ is the composition of $\mathcal{V}$ of with a H\"older function. More precisely, we have
\[
U_{x_j}\simeq |\mathcal{V}_j|^\frac{2}{p},
\]
thus if $1-s<2/p$ we get
\[
\begin{split}
\|U_{x_j}\|_{W^{1-s,p}}^p&\lesssim \sup_{|h|>0}\int \frac{|U_{x_j}(\cdot+h)-U_{x_j}|^p}{|h|^2}\,dx \\
&\lesssim \sup_{|h|>0}\int \frac{|\mathcal{V}_j(\cdot+h)-\mathcal{V}_{j}|^2}{|h|^2}\,dx \simeq \int \left|\nabla \mathcal{V}_j\right|^2\,dx\lesssim\int \left| \mathcal{V}_{x_j}\right|^2\,dx.
\end{split}
\]
This would permit to obtain the desired estimate on $\mathcal{V}$, under the standing assumption on $f$.

\begin{oss} 
Actually, the genesis of Theorem \ref{teo:local} is somehow different from the
above sketched proof. Indeed, the fact that fractional Sobolev regularity of $f$ should be enough to obtain $\mathcal{V}\in W^{1,2}$ appeared as natural in the framework of the {\it regularity via duality} strategy presented in \cite{LNRegDual}. This is a general strategy, first used in the much harder context of variational methods for the incompressible Euler equation by Brenier in \cite{br}, and then re-applied to Mean-Field Games (see for instance \cite{ProSan} for a simple case where this strategy is easy to understand). This strategy allows to prove estimates on the incremental ratios 
\[
\frac{u(\cdot+h)-u}{h},
\] 
of the solutions $u$ of convex variational problems by using the non-optimal function $u(\cdot+h)$ in the corresponding primal-dual optimality conditions.
\end{oss}
In our main result above we make the restriction $p>2$. For completeness, let us comment on the sub-quadratic case.
\begin{oss}[The case $1<p<2$] The sub-quadratic case is simpler, indeed we already know that in this case 
\[
f\in L^{p'}_{\rm loc}(\Omega)\quad \Longrightarrow \quad \mathcal{V}\in W^{1,2}_{\rm loc}(\Omega;\mathbb{R}^N)\quad \Longrightarrow\quad \nabla U\in W^{1,p}_{\rm loc}(\Omega;\mathbb{R}^N),
\] 
see for example\footnote{In \cite{de} as well the result is stated directly for $\nabla U$. However, it is easily seen that the very same proof leads to the stronger result for $\mathcal{V}$.} \cite[Theorem]{de} by de Thelin. 
We can have an idea of the proof of this result
by still following the guidelines sketched above. We treat the left-hand side of \eqref{equazionederivata} as before, while on the right-hand side one now performs an integration by parts and uses H\"older's inequality. These yield
\[
\left|\int f_{x_j}\,U_{x_j}\,dx\right|\le \|f\|_{L^{p'}}\,\|U_{x_j\,x_j}\|_{L^p}\qquad \mbox{ and }\qquad \|U_{x_j\,x_j}\|_{L^p}\lesssim \|\nabla \mathcal{V}\|_{L^2}\,\|\nabla U\|_{L^p}^\frac{2-p}{2}, 
\]
and the term containing the gradient of $\nabla \mathcal{V}$ can be absorbed in the left-hand side. 
Observe that
\[
p'\searrow 2 \quad \mbox{ as }\quad p\nearrow 2,
\]
thus again the assumption on $f$ is consistent with the case of the Laplacian.
\end{oss}

\subsection{Plan of the paper}

We set the notation and recall the basic facts on functional spaces in Section \ref{sec:2}. Here, the important point is Theorem \ref{teo:filipponecas}, which proves inequality \eqref{voto}. In Section \ref{sec:3} we consider a regularization of equation \eqref{plaplace} and prove a Sobolev estimate, independent of the regularization parameter (Proposition \ref{prop:sobolevuniforme}). Then in Section \ref{sec:4} we show how to take the estimate to the limit and achieve the proof of Theorem \ref{teo:local}. We show with an example that our assumption is essentially sharp: this is Section \ref{sec:5}. The paper closes with an appendix containing some technical tools needed for the proof of Theorem \ref{teo:filipponecas}.

\begin{ack}
Both authors have been supported by the {\it Agence Nationale de la Recherche}, through the project ANR-12-BS01-0014-01 {\sc Geometrya}.
The first author is a member of the {\it Gruppo Nazionale per l'Analisi Matematica, la Probabilit\`a
e le loro Applicazioni} (GNAMPA) of the {\it Istituto Nazionale di Alta Matematica} (INdAM).
\end{ack}

\section{Preliminaries}
\label{sec:2}

\subsection{Notation}
For a measurable function $\psi:\mathbb{R}^N\to \mathbb{R}^k$ and a vector $h\in\mathbb{R}^N$, we define
\[
\psi_h(x)=\psi(x+h),\qquad \delta_h \psi(x)=\psi_h(x)-\psi(x),
\]
and
\[
\delta_h^2 \psi(x)=\delta_h(\delta_h \psi(x))=\psi_{2h}(x)+\psi(x)-2\,\psi_h(x).
\]
We consider the two vector-valued functions
\[
\nabla G(z)=|z|^{p-2}\,z \qquad \mbox{ and }\qquad V(z)=|z|^\frac{p-2}{2}\,z,\qquad \mbox{ for }z\in\mathbb{R}^N.
\]
The following inequalities are well-known, we omit the proof.
\begin{lemma}
Let $p>2$, for every $z,w\in\mathbb{R}^N$ we have
\begin{equation}
\label{1}
|z-w|\le C_1\,|V(z)-V(w)|^\frac{2}{p},
\end{equation}
%\begin{equation}
%\label{2}
%\langle \nabla G(z)-\nabla G(w),z-w\rangle\ge C_2\,|V(z)-V(w)|^2,
%\end{equation}
\begin{equation}
\label{3}
\Big|V(z)-V(w)\Big|\le C_2\, \left(|z|^\frac{p-2}{2}+|w|^\frac{p-2}{2}\right) |z-w|,
\end{equation}
%and
%\begin{equation}
%\label{4}
%|\nabla G(z)-\nabla G(w)|\le C_4\,\left(|z|^\frac{p-2}{2}+|w|^\frac{p-2}{2}\right)\,|V(z)-V(w)|,
%\end{equation}
for some $C_1=C_1(p)>0$ and $C_2=C_2(p)>0$.
\end{lemma}

\subsection{Functional spaces}
We recall the definition of some fractional Sobolev spaces needed in the sequel. 
Let $1\le q<\infty$ and let $\psi\in L^q(\mathbb{R}^N)$, for $0<\beta\le 1$ we set
\[
[\psi]_{\mathcal{N}^{\beta,q}_\infty(\mathbb{R}^N)}:=\sup_{|h|>0} \left\|\frac{\delta_h \psi}{|h|^{\beta}}\right\|_{L^q(\mathbb{R}^N)},
\]
and for $0<\beta<2$
\[
[\psi]_{\mathcal{B}^{\beta,q}_\infty(\mathbb{R}^N)}:=\sup_{|h|>0} \left\|\frac{\delta_h^2 \psi}{|h|^{\beta}}\right\|_{L^q(\mathbb{R}^N)}.
\]
We then introduce the two Besov-type spaces
\[
\mathcal{N}^{\beta,q}_\infty(\mathbb{R}^N)=\left\{\psi\in L^q(\mathbb{R}^N)\, :\, [\psi]_{\mathcal{N}^{\beta,q}_\infty(\mathbb{R}^N)}<+\infty\right\},\qquad 0<\beta\le 1,
\]
and
\[
\mathcal{B}^{\beta,q}_\infty(\mathbb{R}^N)=\left\{\psi\in L^q(\mathbb{R}^N)\, :\, [\psi]_{\mathcal{B}^{\beta,q}_\infty(\mathbb{R}^N)}<+\infty\right\},\qquad 0<\beta<2.
\]
We also need the {\it Sobolev-Slobodecki\u{\i} space}
\[
W^{\beta,q}(\mathbb{R}^N)=\left\{\psi\in L^q(\mathbb{R}^N)\, :\, [\psi]_{W^{\beta,q}(\mathbb{R}^N)}<+\infty\right\},\qquad 0<\beta<1,
\]
where the seminorm $[\,\cdot\,]_{W^{\beta,q}(\mathbb{R}^N)}$ is defined by
\[
[\psi]_{W^{\beta,q}(\mathbb{R}^N)}=\left(\int_{\mathbb{R}^N}\int_{\mathbb{R}^N} \frac{|\psi(x)-\psi(y)|^q}{|x-y|^{N+\beta\,q}}\,dx\,dy\right)^\frac{1}{q}.
\]
For $1<\beta<2$, the space $W^{\beta,q}(\mathbb{R}^N)$ consists of
\[
W^{\beta,q}(\mathbb{R}^N)=\left\{\psi\in L^q(\mathbb{R}^N)\, :\, \nabla\psi\in W^{\beta-1,q}(\mathbb{R}^N)\right\}.
\]
More generally, if $\Omega\subset\mathbb{R}^N$ is an open set, the space $W^{\beta,q}(\Omega)$ is defined by
\[
W^{\beta,q}(\Omega)=\left\{\psi\in L^q(\Omega)\, :\, [\psi]_{W^{\beta,q}(\Omega)}<+\infty\right\},\qquad 0<\beta<1,
\]
and the seminorm $[\,\cdot\,]_{W^{\beta,q}(\Omega)}$ is defined accordingly. We endow this space with the norm
\[
\|\psi\|_{W^{\beta,q}(\Omega)}=\|\psi\|_{L^q(\Omega)}+[\psi]_{W^{\beta,q}(\Omega)}.
\]
For an open bounded set $\Omega\subset\mathbb{R}^N$ and $0<\beta\le 1$, the {\it homogeneous Sobolev-Slobodecki\u{\i} space} $\mathcal{D}^{\beta,q}_0(\Omega)$ is defined as the completion of $C^\infty_0(\Omega)$ with respect to the norm
\[
\psi\mapsto \|\psi\|_{\mathcal{D}^{\beta,q}_0(\Omega)}:=[\psi]_{W^{\beta,q}(\mathbb{R}^N)}\qquad \mbox{ if } 0<\beta<1,
\] 
or
\[
\psi\mapsto \|\psi\|_{\mathcal{D}^{1,q}_0(\Omega)}:=\|\nabla \psi\|_{L^q(\mathbb{R}^N)}\qquad \mbox{ if } \beta=1.
\]
Finally, for $0<\beta\le 1$ the topological dual of $\mathcal{D}^{\beta,q}_0(\Omega)$ will be denoted by $\mathcal{D}^{-\beta,q'}(\Omega)$. We endow this space with the natural dual norm, defined by
\[
\|F\|_{W^{-\beta,q'}(\Omega)}=\sup\left\{|\langle F,\varphi\rangle|\, :\, \varphi\in C^\infty_0(\Omega)\ \mbox{ with }\ \|\psi\|_{\mathcal{D}^{\beta,q}_0(\Omega)}\le 1\right\},\qquad F\in \mathcal{D}^{-\beta,q'}(\Omega).
\]
\begin{oss}
\label{oss:poincare}
It is not difficult to show that for $0<\beta\le 1$ we have
\[
\mathcal{D}^{\beta,q}_0(\Omega)\hookrightarrow L^q(\Omega),
\]
thanks to Poincar\'e inequality. The latter reads as
\[
\|u\|_{L^q(\Omega)}^q\le C\,|\Omega|^\frac{\beta\,q}{N}\,\|u\|^q_{\mathcal{D}^{\beta,q}_0(\Omega)},
\]
with $C=C(N,\beta,q)>0$.
\end{oss}

\subsection{Embedding results}
In order to make the paper self-contained, we recall some functional inequalities needed in dealing with fractional Sobolev spaces.
\begin{lemma}
Let $0<\beta<1$ and $1\le q<\infty$, then we have the continuous embedding 
\[
\mathcal{B}^{\beta,q}_\infty(\mathbb{R}^N)\hookrightarrow \mathcal{N}^{\beta,q}_\infty(\mathbb{R}^N).
\] 
More precisley, for every $\psi\in \mathcal{B}^{\beta,q}_\infty(\mathbb{R}^N)$ we have
\[
[\psi]_{\mathcal{N}^{\beta,q}_\infty(\mathbb{R}^N)}\le \frac{C}{1-\beta}\,[\psi]_{\mathcal{B}^{\beta,q}_\infty(\mathbb{R}^N)},
\]
for some constant $C=C(N,q)>0$. 
\end{lemma}
\begin{proof}
We already know that 
\[
[\psi]_{\mathcal{N}^{\beta,q}_\infty(\mathbb{R}^N)}\le \frac{C}{1-\beta}\,\left[[\psi]_{\mathcal{B}^{\beta,q}_\infty(\mathbb{R}^N)}+\|\psi\|_{L^q(\mathbb{R}^N)}\right],
\]
se for example \cite[Lemma 2.3]{brolin}.
In particular, by using this inequality for $\psi^\lambda(x)=\psi(\lambda\,x)$ with $\lambda>0$, after a change of variable we obtain
\[
\lambda^{\beta-\frac{N}{q}}\,[\psi]_{\mathcal{N}^{\beta,q}_\infty(\mathbb{R}^N)}\le \frac{C}{1-\beta}\,\left[\lambda^{\beta-\frac{N}{q}}\,[\psi]_{\mathcal{B}^{\beta,q}_\infty(\mathbb{R}^N)}+\lambda^{-\frac{N}{q}}\,\|\psi\|_{L^q(\mathbb{R}^N)}\right].
\]
By multiplying by $\lambda^{N/q-\beta}$ and taking the limit as $\lambda$ goes to $+\infty$, we get the desired inequality.
\end{proof}
\begin{prop}
\label{prop:lostinpassing}
Let $1\le q<\infty$ and $0<\alpha<\beta< 1$. We have the continuous embedding 
\[
\mathcal{N}^{\beta,q}_\infty(\mathbb{R}^N)\hookrightarrow W^{\alpha,q}(\mathbb{R}^N).
\]
More precisely, for every $\psi\in \mathcal{N}^{\beta,q}_\infty(\mathbb{R}^N)$ we have
\[
[\psi]^q_{W^{\alpha,q}(\mathbb{R}^N)}\le C\,\frac{\beta}{(\beta-\alpha)\,\alpha}\, \left([\psi]_{\mathcal{N}^{\beta,q}_\infty(\mathbb{R}^N)}^q\right)^\frac{\alpha}{\beta}\,\left(\|\psi\|^q_{L^q(\mathbb{R}^N)}\right)^\frac{\beta-\alpha}{\beta},
\]
for some constant $C=C(N,q)>0$.
\end{prop}
\begin{proof}
Let us fix $h_0>0$, by appealing for example to \cite[Proposition 2.7]{brolin}, we already know that
\[
[\psi]^q_{W^{\alpha,q}(\mathbb{R}^N)}\le C\,\left(\frac{h_0^{(\beta-\alpha)\,q}}{\beta-\alpha}\, \sup_{0<|h|<h_0} \left\|\frac{\delta_h \psi}{|h|^{\beta}}\right\|_{L^q(\mathbb{R}^N)}^q+\frac{h_0^{-\alpha\,q}}{\alpha}\,\|\psi\|^q_{L^q(\mathbb{R}^N)}\right).
\]
for a constant $C$ depending on $N$ and $q$ only. If we now optimize the right-hand side in $h_0$ we get the desired conclusion.
\end{proof}
\begin{prop}
%\label{prop:yes}
Let $1\le q<\infty$ and $0<\alpha<\beta<1$. We have the continuous embedding 
\[
\mathcal{B}^{1+\beta,q}_{\infty}(\mathbb{R}^N)\hookrightarrow W^{1+\alpha,q}(\mathbb{R}^N).
\] 
In particular, for every $\psi\in \mathcal{B}^{1+\beta,q}_\infty(\mathbb{R}^N)$ we have $\nabla \psi\in W^{\alpha,q}(\mathbb{R}^N)$, with the following estimates
\begin{equation}
\label{02081980}
\|\nabla \psi\|^q_{L^q(\mathbb{R}^N)}\le \frac{C}{\beta^\frac{\beta+q}{\beta+1}}\,\left(\|\psi\|^q_{L^q(\mathbb{R}^N)}\right)^\frac{\beta}{\beta+1}\, \left([\psi]^q_{\mathcal{B}^{1+\beta,q}_\infty(\mathbb{R}^N)}\right)^\frac{1}{\beta+1},
\end{equation}
for some $C=C(N,q)>0$, and
\begin{equation}
\label{02082015}
[\nabla \psi]^q_{W^{\alpha,q}(\mathbb{R}^N)}\le C\left( [\psi]^q_{\mathcal{B}^{1+\beta,q}_\infty(\mathbb{R}^N)}\right)^\frac{\alpha+1}{\beta+1}\,\left(\|\psi\|^q_{L^q(\mathbb{R}^N)}\right)^\frac{\beta-\alpha}{\beta+1},
\end{equation}
for some $C=C(N,q,\alpha,\beta)>0$, which blows up as $\alpha\nearrow \beta$, $\beta\searrow 0$ or $\beta\nearrow 1$.
\end{prop}
\begin{proof}
We already know that
\[
\|\nabla \psi\|^q_{L^q(\mathbb{R}^N)}\le C\,\|\psi\|^q_{L^q(\mathbb{R}^N)}+\frac{C}{\beta^q}\, [\psi]^q_{\mathcal{B}^{1+\beta,q}_\infty(\mathbb{R}^N)},
\]
see for example \cite[Proposition 2.4]{brolin}.
By replacing $\psi$ with the rescaled function $\psi^\lambda(x)=\psi(\lambda\,x)$ and optimizing in $\lambda$, we get \eqref{02081980}.
\par
For the second estimate \eqref{02082015}, it is sufficient to observe that
\[
[\nabla \psi]^q_{\mathcal{N}^{\beta,q}_\infty(\mathbb{R}^N)}\le \frac{C}{\beta^q\,(1-\beta)^q} [\psi]^q_{\mathcal{B}^{1+\beta,q}_\infty(\mathbb{R}^N)},
\]
again by \cite[Proposition 2.4]{brolin}.
Then by using Proposition \ref{prop:lostinpassing} for $\nabla \psi$ and \eqref{02081980}, we get the desired conclusion. 
\end{proof}

\subsection{An inequality for negative norms}
As explained in the Introduction, a crucial r\^ole in the proof of Theorem \ref{teo:local} is played by the following weak generalization of the so-called {\it Ne\v{c}as' negative norm Theorem} (which corresponds to $\beta=0$, see \cite{Ne}).
\begin{teo}
\label{teo:filipponecas}
Let $0<\beta<1$ and $1<q<\infty$. Let $B\subset\mathbb{R}^N$ be an open ball, for every $f\in W^{\beta,q}(B)$ we have
\begin{equation}
\label{filipponecasSI}
\|f_{x_j}\|_{\mathcal{D}^{\beta-1,q}(B)}\le C\,[f]_{W^{\beta,q}(B)},\qquad j=1,\dots,N,
\end{equation}
for a constant $C=C(N,\beta,q)>0$.
\end{teo}
\begin{proof}
We explain the guidelines of the proof, by referring the reader to Appendix \ref{sec:A} for the missing details. We first prove 
\begin{equation}
\label{filipponecas}
\|f_{x_j}\|_{\mathcal{D}^{\beta-1,q}(B)}\le C\,\left(\|f\|_{L^q(B)}+[f]_{W^{\beta,q}(B)}\right),\qquad j=1,\dots,N.
\end{equation}
This estimate says that the linear operator
\[
T_j:W^{\beta,q}(B)\to \mathcal{D}^{\beta-1,q}(B),
\] 
defined by the weak $j-$th derivative is continuous. It is easy to see that $T_j$ is continuous as an operator defined on $W^{1,q}(B)$ and $L^q(B)$. More precisely, the following operators
\[
\begin{array}{clll}
T_j:&W^{1,q}(B)&\to& L^q(B)\\
&&&\\
T_j:&L^q(B)&\to& \mathcal{D}^{-1,q}(B)
\end{array}
\]
are linear and continuous. In other words, inequality \eqref{filipponecas} is true for the extremal cases $\beta=0$ and $\beta=1$. We observe that $W^{\beta,q}(B)$ is an interpolation space between $W^{1,q}(B)$ and $L^q(B)$, i.e.
\[
W^{\beta,q}(B)=\left(L^q(B),W^{1,q}(B)\right)_{\beta,q},
\] 
see Definition \ref{defi:inter} for the notation. We can then obtain quite easily that $T_j$ is continuous from $W^{\beta,q}(B)$ to the interpolation space between $\mathcal{D}^{-1,q}(B)$ and $L^q(B)$, i.e. 
\[
\left(\mathcal{D}^{-1,q}(B),L^q(B)\right)_{\beta,q},
\] 
see Lemma \ref{lm:3}. Such a space can be computed explicitely: as one may expect, it coincides with the dual Sobolev-Slobodecki\u{\i} space $\mathcal{D}^{\beta-1,q}(B)$ (see Lemma \ref{lm:4}). This proves inequality \eqref{filipponecas}.
\par
In order to get \eqref{filipponecasSI} and conclude, it is now sufficient to use a standard scaling argument.
Let us assume for simplicity that $B$ is centered at the origin, for $f\in W^{\beta,q}(B)$ and every $\lambda>0$, we define $f_\lambda(x)=f(x/\lambda)$. This belongs to $W^{\beta,q}(\lambda\,B)$, then from \eqref{filipponecas} and the scaling properties of the norms, we get
\[
\begin{split}
\lambda^{\frac{N}{q}-\beta}\,\|f_{x_j}\|_{\mathcal{D}^{\beta-1,q}(B)}=\|(f_\lambda)_{x_j}\|_{\mathcal{D}^{\beta-1,q}(\lambda\,B)}&\le C\,\left(\|f_\lambda\|_{L^q(\lambda\,B)}+[f_\lambda]_{W^{\beta,q}(\lambda\,B)}\right)\\
&=C\,\left(\lambda^\frac{N}{q}\,\|f\|_{L^q(B)}+\lambda^{\frac{N}{q}-\beta}\,[f]_{W^{\beta,q}(B)}\right).
\end{split}
\]
If we multiply by $\lambda^{\beta-N/q}$ and then let $\lambda$ go to $0$, we finally get the desired estimate.
\end{proof}

\section{Estimates for a regularized problem}
\label{sec:3}

Let $U$ and $f$ be as in the statement of Theorem \ref{teo:local}. Let $B\Subset \Omega$ be an open ball, for every $\varepsilon>0$ we consider the problem
\begin{equation}
\label{approssimato}
\left\{\begin{array}{rcll}
-\mathrm{div}\nabla G_\varepsilon(\nabla u)&=&f_\varepsilon,&\mbox{ in } B,\\
u&=&U,&\mbox{ on }\partial B,
\end{array}
\right.
\end{equation}
where:
\begin{itemize}
\item $G_\varepsilon(z)=\dfrac{1}{p}\,(\varepsilon+|z|^2)^\frac{p}{2}$, for every $z\in\mathbb{R}^N$;
\vskip.2cm
\item $f_\varepsilon=f\ast \varrho_\varepsilon$ and $\{\varrho_\varepsilon\}_{\varepsilon>0}$ is a family of standard compactly supported $C^\infty$ mollifiers.
\vskip.2cm
\end{itemize} 
Problem \eqref{approssimato} admits a unique solution $u_\varepsilon\in W^{1,p}(B)$, which is locally smooth in $B$ by standard elliptic regularity.

\begin{prop}[Uniform energy estimate]
With the notation above, we have
\begin{equation}
\label{epsilon}
\begin{split}
\int_B |\nabla u_\varepsilon|^p\,dx&\le C\,\varepsilon^\frac{p-1}{2}\,\int_B|\nabla U|\,dx+C\,\int_B |\nabla U|^{p}\,dx+C\,|B|^\frac{p'}{N}\,\|f_\varepsilon\|_{L^{p'}(B)}^{p'},
\end{split}
\end{equation}
for some $C=C(N,p)>0$.
\end{prop}
\begin{proof}
The proof is standard, we include it for completeness. We take the weak formulation of \eqref{approssimato}
\[
\int \langle \nabla G_\varepsilon(\nabla u_\varepsilon),\nabla\varphi\rangle\,dx=\int f_\varepsilon\,\varphi\,dx,\qquad \mbox{ for every }\varphi\in W^{1,p}_0(B),
\]
and insert the test function $\varphi=u_\varepsilon-U$. This gives
\[
\begin{split}
\int_B \langle \nabla G_\varepsilon(\nabla u_\varepsilon),\nabla u_\varepsilon\rangle\,dx&=\int_B \langle \nabla G_\varepsilon(\nabla u_\varepsilon),\nabla U\rangle\,dx+\int_B f_\varepsilon\,(u_\varepsilon-U)\,dx\\
&\le \int |\nabla G_\varepsilon(\nabla u_\varepsilon)|\,|\nabla U|\,dx+\|f_\varepsilon\|_{L^{p'}(B)}\,\|u_\varepsilon-U\|_{L^p(B)}.
\end{split}
\]
We then observe that 
\[
\langle\nabla G_\varepsilon(z),z\rangle\ge |z|^p,\qquad z\in\mathbb{R}^N,
\]
and
\[
|\nabla G_\varepsilon(z)|\le C\, \left(\varepsilon^\frac{p-2}{2}\,|z|+|z|^{p-1}\right)\le C\,\varepsilon^\frac{p-1}{2}+2\,C\,|z|^{p-1},\qquad z\in\mathbb{R}^N,
\]
for some $C=C(p)>0$. By using these inequalities, we get
\[
\begin{split}
\int_B |\nabla u_\varepsilon|^p\,dx&\le C\,\varepsilon^\frac{p-1}{2}\int_B|\nabla U|\,dx+C\int_B |\nabla u_\varepsilon|^{p-1}\,|\nabla U|\,dx+\|f_\varepsilon\|_{L^{p'}(B)}\,\|u_\varepsilon-U\|_{L^p(B)}\\
&\le C\,\varepsilon^\frac{p-1}{2}\,\int_B|\nabla U|\,dx+C\,\tau\,\int_B |\nabla u_\varepsilon|^{p}\,dx\\
&+C\,\tau^{-\frac{1}{p-1}}\,\int_B |\nabla U|^{p}\,dx+C\,|B|^\frac{1}{N}\,\|f_\varepsilon\|_{L^{p'}(B)}\,\|\nabla u_\varepsilon-\nabla U\|_{L^p(B)},
\end{split}
\]
where we used Young's inequality in the second term and Poincar\'e's inequality in the last one. By taking $\tau>0$ sufficiently small, we can then obtain
\[
\begin{split}
\int_B |\nabla u_\varepsilon|^p\,dx
&\le C\,\varepsilon^\frac{p-1}{2}\,\int_B|\nabla U|\,dx+C\,\int_B |\nabla U|^{p}\,dx\\
&+C\,|B|^\frac{1}{N}\,\|f_\varepsilon\|_{L^{p'}(B)}\,\|\nabla u_\varepsilon-\nabla U\|_{L^p(B)}.
\end{split}
\]
We can now use the triangle inequality on the last term and conclude with a further application of Young's inequality, in order to absorb the term containing $\nabla u_\varepsilon$. We leave the details to the reader.
\end{proof}
The proof of Theorem \ref{teo:local} is crucially based on the following
\begin{prop}[Uniform Sobolev estimate]
\label{prop:sobolevuniforme}
We set
\[
\mathcal{V}_\varepsilon:=V(\nabla u_\varepsilon)=|\nabla u_\varepsilon|^\frac{p-2}{2}\,\nabla u_\varepsilon.
\]
Let $(p-2)/p<s\le 1$, for every pair of concentric balls $B_r\Subset B_R\Subset B$ and every $j=1,\dots,N$ we have
\begin{equation}
\label{Vepsilon}
\begin{split}
\int_{B_r} \left|\left(\mathcal{V}_\varepsilon\right)_{x_j}\right|^2\,dx&\le \frac{C}{(R-r)^2}\int_{B_R} (\varepsilon+|\nabla u_\varepsilon|^2)^\frac{p}{2}\,dx+C\left(R^{\left(s-\frac{p-2}{p}\right)}\,[f_\varepsilon]_{W^{s,p'}(B_R)}\right)^{p'},
\end{split}
\end{equation}
for a constant $C=C(N,p,s)>0$ which blows up as $s\searrow (p-2)/p$.
\end{prop}
\begin{proof}
In what follows, for notational simplicity we omit to indicate the dependence on $\varepsilon>0$ and simply write
\[
G,\quad u,\quad \mathcal{V}\quad \mbox{ and }\quad f.
\]
For $j\in\{1,\dots,N\}$, in the weak formulation of \eqref{approssimato} we insert a test function of the form $\varphi_{x_j}$ for $\varphi\in C^\infty_0(B)$. After an integration by parts, we obtain
\[
\int \langle D^2 G(\nabla u)\,\nabla u_{x_j},\nabla \varphi\rangle\,dx=\int f_{x_j}\,\varphi\,dx.
\]
By density, this relation remains true for $\varphi\in W^{1,p}$ with compact support in $B$. We then insert the test function
\[
\varphi=\zeta^2\, u_{x_j},
\]
with $\zeta\in C^\infty_0(B_R)$ a standard cut-off function such that
\begin{equation}
\label{zzeta}
0\le \zeta\le 1,\qquad \zeta\equiv 1\ \mbox{ in } B_r,\qquad |\nabla \zeta|\le \frac{C}{R-r}.
\end{equation}
Thus we obtain
\begin{equation}
\label{apparecchio}
\begin{split}
\int \langle D^2 G(\nabla u)\,\nabla u_{x_j},\nabla u_{x_j}\rangle\,\zeta^2\,dx&=-2\,\int\langle D^2 G(\nabla u)\,\nabla u_{x_j},\nabla \zeta\rangle\,u_{x_j}\,\zeta\,dx\\
&+ \int f_{x_j}\,u_{x_j}\,\zeta^2\,dx.
\end{split}
\end{equation}
We first observe that the left-hand side is positive, since $G$ is convex. As for the right-hand side, we have
\[
\begin{split}
\int \langle D^2 G(\nabla u)\,\nabla u_{x_j},\nabla \zeta\rangle\,u_{x_j}\,\zeta\,dx
&\le \int \left|\langle D^2 G(\nabla u)\,\nabla u_{x_j},\nabla \zeta\rangle\right|\,|u_{x_j}|\,\zeta\,dx\\
&\le \int \sqrt{\langle D^2 G(\nabla u)\,\nabla u_{x_j},\nabla u_{x_j}\rangle}\\
&\times\sqrt{\langle D^2 G(\nabla u)\,\nabla \zeta,\nabla \zeta\rangle} \,|u_{x_j}|\,\zeta\,dx
\end{split},
\]
thanks to Cauchy-Schwartz inequality. By using Young's inequality in a standard fashion, from \eqref{apparecchio} and the previous inequality we can obtain
\[
\begin{split}
\int \langle D^2 G(\nabla u)\,\nabla u_{x_j},\nabla u_{x_j}\rangle\,\zeta^2\,dx&\le 4\,\int\langle D^2 G(\nabla u)\,\nabla\zeta,\nabla \zeta\rangle\,|u_{x_j}|^2\,dx\\
&+ 2\,\int f_{x_j}\,u_{x_j}\,\zeta^2\,dx.
\end{split}
\]
We now use that
\[
|z|^{p-2}\,|\xi|^2\le \langle D^2 G(z)\,\xi,\xi\rangle\le (p-1)\,(\varepsilon+|z|^2)^\frac{p-2}{2}\,|\xi|^2,\qquad z,\xi\in\mathbb{R}^N,
\]
thus with simple manipulations we get
\[
\begin{split}
\int |\nabla u|^{p-2}\,|\nabla u_{x_j}|^2\,\zeta^2\,dx&\le \frac{4\,(p-1)}{(R-r)^2}\int_{B_R} (\varepsilon+|\nabla u|^2)^\frac{p}{2}\,dx+ 2\,\int f_{x_j}\,u_{x_j}\,\zeta^2\,dx.
\end{split}
\]
We observe that
\begin{equation}
\label{Vona}
|\nabla u|^{p-2}\,|\nabla u_{x_j}|^2=\frac{4}{p^2}\,\left|\left(|\nabla u|^\frac{p-2}{2}\,\nabla u\right)_{x_j}\right|^2=\frac{4}{p^2}\,\left|\mathcal{V}_{x_j}\right|^2,
\end{equation}
thus we have
\begin{equation}
\label{rullo!}
\int \left|\mathcal{V}_{x_j}\right|^2\,\zeta^2\,dx\le \frac{(p-1)\,p^2}{(R-r)^2}\int_{B_R} (\varepsilon+|\nabla u|^2)^\frac{p}{2}\,dx+\frac{p^2}{2}\,\int f_{x_j}\,u_{x_j}\,\zeta^2\,dx.
\end{equation}
We are left with the estimate of the term containing $f_{x_j}$. By definition of dual norm, for\footnote{We exclude here the case $s=1$, since this is the easy case. It would be sufficient to use H\"older's inequality with exponents $p$ and $p'$ to conclude.}  $(p-2)/p<s<1$ we have
\begin{equation}
\label{emmo?}
\int f_{x_j}\,u_{x_j}\,\zeta^2\,dx\le \|f_{x_j}\|_{\mathcal{D}^{s-1,p'}(B_R)}\,[u_{x_j}\,\zeta^2]_{W^{1-s,p}(\mathbb{R}^N)}.
\end{equation}
Observe that by assumption $0<1-s<2/p$, thus by Proposition \ref{prop:lostinpassing} with $\alpha=1-s$ and $\beta=2/p$, we have
\begin{equation}
\label{bleah}
\begin{split}
\left[u_{x_j}\,\zeta^2\right]^p_{W^{1-s,p}(\mathbb{R}^N)}&\le \frac{C}{\left(s-\dfrac{p-2}{p}\right)\,(1-s)}\, \left(\left[u_{x_j}\,\zeta^2\right]_{\mathcal{N}^{\frac{2}{p},p}_\infty(\mathbb{R}^N)}^p\right)^{(1-s)\frac{p}{2}}\,\left(\|u_{x_j}\|^p_{L^p(B_R)}\right)^{1-\frac{p}{2}(1-s)},
\end{split}
\end{equation}
In order to estimate the norm of $u_{x_j}\,\zeta^2$, we recall that
\[
\left[u_{x_j}\,\zeta^2\right]_{\mathcal{N}^{\frac{2}{p},p}_\infty(\mathbb{R}^N)}^p=\sup_{|h|>0} \int_{\mathbb{R}^N} \frac{|\delta_h(u_{x_j}\,\zeta^2)|^p}{|h|^2}\,dx.
\]
Then we observe that from \eqref{1}
\[
|a-b|^p\le C\, \left||a|^\frac{p-2}{2}\,a-|b|^\frac{p-2}{2}\,b\right|^{2},\qquad a,b\in\mathbb{R}.
\]
Thus we obtain
\[
\begin{split}
\int_{\mathbb{R}^N} &\frac{|\delta_h(u_{x_j}\,\zeta^2)|^p}{|h|^2}\,dx\le C\, \int_{\mathbb{R}^N} \frac{\left|\delta_h\left(|u_{x_j}|^\frac{p-2}{2}\,u_{x_j}\,\zeta^p\right)\right|^2}{|h|^2}\,dx\le C\,\int_{\mathbb{R}^N} \left|\nabla \left(|u_{x_j}|^\frac{p-2}{2}\,u_{x_j}\,\zeta^p\right)\right|^2\,dx,
\end{split}
\]
where the second inequality comes from the classical characterization of $W^{1,2}$ in terms of finite differences.
By recalling the properties \eqref{zzeta} of $\zeta$, with simple manipulations we thus obtain
\[
\left[u_{x_j}\,\zeta^2\right]_{\mathcal{N}^{\frac{2}{p},p}_\infty(\mathbb{R}^N)}^p\le C\, \int \left|\nabla \left(|u_{x_j}|^\frac{p-2}{2}\,u_{x_j}\right)\right|^2\,\zeta^2\,dx+\frac{C}{(R-r)^2}\,\int_{B_R} |u_{x_j}|^p\,dx,
\]
for a constant $C=C(N,p)>0$. Observe that
\[
\left|\nabla \left(|u_{x_j}|^\frac{p-2}{2}\,u_{x_j}\right)\right|^2=\frac{p^2}{4}\,|u_{x_j}|^{p-2}\,|\nabla u_{x_j}|^2\le \frac{p^2}{4}\,|\nabla u|^{p-2}\,|\nabla u_{x_j}|^2=|\mathcal{V}_{x_j}|^2,
\]
thanks to \eqref{Vona}. This yields
\[
\left[u_{x_j}\,\zeta^2\right]_{\mathcal{N}^{\frac{2}{p},p}_\infty(\mathbb{R}^N)}^p\le C\, \int |\mathcal{V}_{x_j}|^2\,\zeta^2\,dx+\frac{C}{(R-r)^2}\,\int_{B_R} |\nabla u|^p\,dx.
\]
By inserting this estimate in \eqref{bleah}, from \eqref{emmo?} we get
\[
\begin{split}
\left|\int f_{x_j}\,u_{x_j}\,\zeta^2\,dx\right|&\le C\,\|f_{x_j}\|_{\mathcal{D}^{s-1,p'}(B_R)}\,\left(\|\mathcal{V}_{x_j}\,\zeta\|^\frac{2}{p}_{L^2(B_R)}+\frac{C}{(R-r)^\frac{2}{p}}\,\|\nabla u\|_{L^p(B_R)}\right)^{(1-s)\frac{p}{2}}\\
&\times \left(\|\nabla u\|_{L^p(B_R)}\right)^{1-\frac{p}{2}(1-s)},
\end{split}
\]
for a constant $C=C(N,p,s)>0$, which blows-up as $s\searrow (p-2)/p$. We can still manipulate a bit the previous estimate and obtain
\[
\begin{split}
\left|\int f_{x_j}\,u_{x_j}\,\zeta^2\,dx\right|&\le C\,\|f_{x_j}\|_{\mathcal{D}^{s-1,p'}(B_R)}\,\left(\|\mathcal{V}_{x_j}\,\zeta\|^{1-s}_{L^2(B_R)}+\frac{C}{(R-r)^{1-s}}\,\|\nabla u\|_{L^p(B_R)}^{(1-s)\frac{p}{2}}\right)\\
&\times \left(\|\nabla u\|_{L^p(B_R)}\right)^{1-\frac{p}{2}(1-s)}\\
&\le  C\,\|f_{x_j}\|_{\mathcal{D}^{s-1,p'}(B_R)}\,\|\mathcal{V}_{x_j}\,\zeta\|^{1-s}_{L^2(B_R)}\,\left(\|\nabla u\|_{L^p(B_R)}\right)^{1-\frac{p}{2}(1-s)}\\
&+\frac{C}{(R-r)^{1-s}}\,\|f_{x_j}\|_{\mathcal{D}^{s-1,p'}(B_R)}\,\|\nabla u\|_{L^p(B_R)},
\end{split}
\]
for a different constant $C>0$, still depending on $N,p$ and $s$ only. We now go back to \eqref{rullo!} and use the previous estimate. This gives
\[
\begin{split}
\int \left|\mathcal{V}_{x_j}\right|^2\,\zeta^2\,dx&\le \frac{C}{(R-r)^2}\int_{B_R} (\varepsilon+|\nabla u|^2)^\frac{p}{2}\,dx\\
&+C\,\|f_{x_j}\|_{\mathcal{D}^{s-1,p'}(B_R)}\,\|\mathcal{V}_{x_j}\,\zeta\|^{1-s}_{L^2(B_R)}\,\left(\|\nabla u\|_{L^p(B_R)}\right)^{1-\frac{p}{2}(1-s)}\\
&+\frac{C}{(R-r)^{1-s}}\,\|f_{x_j}\|_{\mathcal{D}^{s-1,p'}(B_R)}\,\|\nabla u\|_{L^p(B_R)}.
\end{split}
\]
We need to absorb the higher order term containing $\mathcal{V}$ in the right-hand side.
For this, we use Young's inequality with exponents
\[
p',\quad \frac{2}{1-s} \quad \mbox{ and }\quad \frac{2\,p}{2-p\,(1-s)},
\]
so to get
\[
\begin{split}
\|f_{x_j}\|_{\mathcal{D}^{s-1,p'}(B_R)}&\|\mathcal{V}_{x_j}\,\zeta\|^{1-s}_{L^2(B_R)}\,\left(\|\nabla u\|_{L^p(B_R)}\right)^{1-\frac{p}{2}(1-s)}\\
&\le C\,\tau^{-\frac{(1-s)}{2}\,p'}\,(R-r)^{\left(s-\frac{p-2}{p}\right)\,p'}\,\|f_{x_j}\|^{p'}_{\mathcal{D}^{s-1,p'}(B_R)}+\tau\,\|\mathcal{V}_{x_j}\,\zeta\|^2_{L^2(B_R)}\\
&+\frac{C}{(R-r)^2}\,\|\nabla u\|_{L^p(B_R)}^p,
\end{split}
\]
which yields (by taking $\tau>0$ small enough)
\[
\begin{split}
\int \left|\mathcal{V}_{x_j}\right|^2\,\zeta^2\,dx&\le \frac{C}{(R-r)^2}\int_{B_R} (\varepsilon+|\nabla u|^2)^\frac{p}{2}\,dx+C\,(R-r)^{\left(s-\frac{p-2}{p}\right)\,p'}\,\|f_{x_j}\|_{\mathcal{D}^{s-1,p'}(B_R)}^{p'}\\
&+\frac{C}{(R-r)^{1-s}}\,\|f_{x_j}\|_{\mathcal{D}^{s-1,p'}(B_R)}\,\|\nabla u\|_{L^p(B_R)}.
\end{split}
\]
We now apply Young's inequality once more and Theorem \ref{teo:filipponecas}, so to obtain \eqref{epsilon}.
\end{proof}

\section{Proof of Theorem \ref{teo:local}}
\label{sec:4}

Let $U\in W^{1,p}_{\rm loc}(\Omega)$ be a local weak solution in $\Omega$ of \eqref{plaplace}. We fix a ball $B_R\Subset\Omega$ and take a pair of concentric ball $B$ and $\widetilde B$ such that $B_R\Subset B\Subset \widetilde B\Subset\Omega$. There exists $\varepsilon_0>0$ such that
\begin{equation}
\label{fepsilon}
\|f_\varepsilon\|_{L^{p'}(B)}+[f_\varepsilon]_{W^{s,p'}(B)}\le \|f\|_{L^{p'}(\widetilde B)}+[f]_{W^{s,p'}(\widetilde B)}<+\infty,\qquad \mbox{ for every }0<\varepsilon<\varepsilon_0.
\end{equation}
We consider for every $0<\varepsilon<\varepsilon_0$ the solution $u_\varepsilon$ of problem \eqref{approssimato} in the ball $B$. By \eqref{epsilon} and \eqref{fepsilon} we know that $\{u_\varepsilon\}_{0<\varepsilon<\varepsilon_0}$ is bounded in $W^{1,p}(B)$, thus by Rellich-Kondra\v{s}ov Theorem we can extract a sequence $\{\varepsilon_k\}_{k\in\mathbb{N}}$ converging to $0$, such that
\[
\lim_{k\to\infty} \|u_{\varepsilon_k}-u\|_{L^p(B)}=0,
\]
for some $u\in W^{1,p}(B)$. Since $u_{\varepsilon_k}$ is the unique solution of
\[
\min\left\{\int_B G_{\varepsilon_k}(\nabla \varphi)\,dx-\int_B f_{\varepsilon_k}\,\varphi\,dx\, :\, \varphi-U\in W^{1,p}_0(B)\right\},
\]
by a standard $\Gamma-$convergence argument we can easily show that $u=U$, i.e. the limit function coincides with our local weak solution $U$. 
\par
From Proposition \ref{prop:sobolevuniforme} and boundedness of $\{u_{\varepsilon_k}\}_{k\in\mathbb{N}}$, we know that $\{\mathcal{V}_{\varepsilon_k}\}_{k\in\mathbb{N}}$ is bounded in $W^{1,2}(B_r;\mathbb{R}^N)$, for every $B_r\Subset B$. Up to passing a subsequence, we can infer convergence to some vector field \(\mathcal{Z}\in W^{1,2}(B_r;\mathbb{R}^N)\), weakly in \(W^{1,2}(B_r;\mathbb{R}^N)\) and strongly in \(L^{2}(B_r;\mathbb{R}^N)\). In particular, this is a Cauchy sequence in $L^2(B_r)$ and by using the elementary inequality \eqref{1},
we obtain that $\{u_{\varepsilon_k}\}_{k\in\mathbb{N}}$ as well is a Cauchy sequence in $W^{1,p}(B_r)$. Thus we obtain 
\[
\lim_{k\to\infty}\left\|\nabla u_{\varepsilon_k}-\nabla U\right\|_{L^p(B_r)}=0.
\]
We need to show that $\mathcal{Z}=|\nabla U|^{(p-2)/2}\,U$. We use the elementary inequality \eqref{3}, this yields
\[
\begin{split}
\int_{B_r} \left||\nabla u_{\varepsilon_k}|^{\frac{p-2}{2}}\,\nabla u_{\varepsilon_k}-|\nabla U|^{\frac{p-2}{2}}\,\nabla U\right|^2\,dx&\le C\, \int_{B_r} \left(|\nabla u_{\varepsilon_k}|^{\frac{p-2}{2}}+|\nabla U|^{\frac{p-2}{2}}\right)^2\,|\nabla u^{\varepsilon_k}-\nabla U|^2\,dx\\
&\le C\,\left( \int_{B_r} \left(|\nabla u_{\varepsilon_k}|^{\frac{p-2}{2}}+|\nabla U|^{\frac{p-2}{2}}\right)^\frac{2\,p}{p-2}\,dx\right)^\frac{p-2}{p}\\
&\times \left(\int_{B_r} |\nabla u_{\varepsilon_k}-\nabla U|^p\,dx\right)^\frac{2}{p}.
\end{split}
\]
By using the strong convergence of the gradients proved above, this gives that $\mathcal{Z}=|\nabla U|^\frac{p-2}{2}\,\nabla U$ and it belongs to $W^{1,2}(B_r;\mathbb{R}^N)$. By arbitrariness of the ball $B_r\Subset B$ in the above discussion, we can now take the ball $B_R$ fixed at the beginning and pass to the limit in \eqref{Vepsilon}, so to obtain the desired estimate \eqref{filipposobolev}.
\par
Finally, the fact that $\nabla U\in W^{\tau,p}_{\rm loc}(\Omega;\mathbb{R}^N)$ follows in standard way from the elementary inequality \eqref{1}. We leave the details to the reader.

\section{An example}
\label{sec:5}
We now show with an explicit example that the assumption on $f$ in Theorem \ref{teo:local} is essentially sharp.
\par 
Let us take $U(x)=|x|^{-\alpha}$, which belongs to $W^{1,p}_{\rm loc}(\mathbb{R}^N)$ if and only if
\[
\alpha<\frac{N}{p}-1.
\]
Then we can compute
\[
|\nabla U|^{p-2}\,\nabla U\simeq |x|^{-(\alpha+1)\,(p-1)-1}\,x,\qquad
\mathrm{div}\left(|\nabla U|^{p-2}\,\nabla U\right)\simeq |x|^{-(\alpha+1)\,(p-1)-1},
\]
and
\[
|\nabla U|^\frac{p-2}{2}\,\nabla U\simeq |x|^{-(\alpha+1)\,\frac{p}{2}-1}\,x.
\]
We observe that
\[
|\nabla U|^\frac{p-2}{2}\,\nabla U\in W^{1,2}_{\rm loc}(\mathbb{R}^N)\quad \Longleftrightarrow\quad \alpha<\frac{N-2}{p}-1=:\widetilde\alpha.
\]
On the other hand the function $f(x)=|x|^{-(\alpha+1)\,(p-1)-1}$ belongs to $W^{s,p'}_{\rm loc}(\mathbb{R}^N)$ if 
\[
(\alpha+1)\,(p-1)+1<\frac{N-s\,p'}{p'}\qquad \mbox{ i.\,e. if }\qquad \alpha<\frac{N}{p}-\frac{s+1}{p-1}-1=:\alpha_s.
\]
If we take $s<(p-2)/p$, we have
\[
\alpha_s>\frac{N-2}{p}-1=\widetilde\alpha,
\]
thus for every such $s$, if we take $\alpha=(N-2)/p-1$ we get
\[
-\Delta_p U=f\in W^{s,p'}_{\rm loc}(\mathbb{R}^N) \qquad \mbox{ and }\qquad |\nabla U|^\frac{p-2}{2}\, \nabla U\not\in W^{1,2}_{\rm loc}(\mathbb{R}^N).
\]
This shows that in Theorem \ref{teo:local} the differentiability exponent $s$ of $f$ can not go below $(p-2)/p$.

\appendix

\section{Tools for the proof of Theorem \ref{teo:filipponecas}}
\label{sec:A}

\subsection{Basics of real interpolation: the $K-$method}
We first present a couple of basic facts from the theory of real interpolation, by referring the reader to \cite{BL} for more details.
\begin{defi}
%\label{defi:compatibile}
Let $X$ and $Y$ two Banach spaces over $\mathbb{R}$. We say that $(X,Y)$ is a {\it compatible couple} is there exists a Hausdorff topological vector space $Z$ such that $X$ and $Y$ are continuously embedded in $Z$.
\end{defi}
When $(X,Y)$ is a compatible couple, it does make sense to consider the two spaces $X\cap Y$ and $X+Y$. Thus we can give the definition of interpolation space.
\begin{defi}
\label{defi:inter}
Let $(X,Y)$ be a compatible couple of Banach spaces over $\mathbb{R}$. For every $u\in X+Y$ and $t>0$ we define
\[
K(t,u,X,Y)=\inf_{u_1\in X,\, u_2\in Y} \Big\{\|u_1\|_X+t\,\|u_2\|_Y\, :\, u=u_1+u_2\Big\}.
\]
For $0<\alpha<1$ and $1<q<\infty$, the {\it interpolation space} $(X,Y)_{\alpha,q}$ consists of
\[
(X,Y)_{\alpha,q}=\left\{u\in X+Y\, :\, \int_0^{+\infty} t^{-\alpha\,q}\,K(t,u,X,Y)^q\,\frac{dt}{t}<+\infty\right\}.
\]
This is a Banach space with the norm
\[
\|u\|_{(X,Y)_{\alpha,q}}=\left(\int_0^{+\infty} t^{-\alpha\,q}\,K(t,u,X,Y)^q\,\frac{dt}{t}\right)^\frac{1}{q}.
\]
\end{defi}
The following result is well-known.
\begin{lemma}
\label{lm:1}
Let $0<\beta<1$ and $1<q<\infty$. For every $B\subset\mathbb{R}^N$ open ball, we have
\[
W^{\beta,q}(B)=\left(L^q(B),W^{1,q}(B)\right)_{\beta,q},
\]
and there exists a constant $C=C(N,\beta,q)>1$ such that
\[
\frac{1}{C}\,\|u\|_{\left(L^q(B),W^{1,q}(B)\right)_{\beta,q}} \le \|u\|_{W^{\beta,q}(B)}\le C\, \|u\|_{\left(L^q(B),W^{1,q}(B)\right)_{\beta,q}}
\]
\end{lemma}
\begin{proof}
This is classical, a proof can be found in \cite[Chapters 34 \& 36]{Ta}.
\end{proof}

\subsection{Weak derivatives}

For $j\in\{1,\dots,N\}$, let us consider the linear operator $T_j$ defined by
the distributional $j-$th partial derivative, i.e.
\[
\langle T_j(F),\varphi\rangle=-\langle F,\varphi_{x_j}\rangle
\]
for every test function $\varphi$ and every distribution $F$. We first recall the following fact, whose proof is straightforward.
\begin{lemma}
%\label{lm:2}
Let $1<q<\infty$ and $j\in \{1,\dots,N\}$, the linear operators 
\[
T_j:W^{1,q}(B)\to L^q(B)\qquad \mbox{ and }\qquad T_j:L^q(B)\to \mathcal{D}^{-1,q}(B),
\] 
are continuous on their domains. More precisely, we have
\begin{equation}
\label{C1}
\|T_j(u)\|_{L^q(B)}\le \|u\|_{W^{1,q}(B)},\qquad \mbox{ for every }u\in W^{1,q}(B),
\end{equation}
and
\begin{equation}
\label{C2}
\|T_j(u)\|_{\mathcal{D}^{-1,q}(B)}\le \|u\|_{L^q(B)},\qquad \mbox{ for every }u\in L^{q}(B).
\end{equation}
\end{lemma}
We can then prove the following result.
\begin{lemma}
\label{lm:3}
Let $0<\beta<1$ and $1<q<\infty$. Then the restriction of $T_j$ to $W^{\beta,q}(B)$ is a linear continuous operator from $W^{\beta,q}(B)$ to the interpolation space
\[
\mathcal{Y}(B):=\left(\mathcal{D}^{-1,q}(B),L^q(B)\right)_{\beta,q}.
\]
In other words, there exist a constant $C=C(N,\beta,q)>0$ such that
\[
\int_0^{+\infty} t^{-\beta\,q}\,K(t,T_j(u),\mathcal{D}^{-1,q}(B),L^q(B))^q\,\frac{dt}{t}\le C\, \|u\|_{W^{\beta,q}(B)}^q,\quad \mbox{ for every } u\in W^{\beta,q}(B).
\]
\end{lemma}
\begin{proof}
Let $u\in W^{\beta,q}(B)$, by Lemma \ref{lm:1} there exists $u_1\in L^q(B)$ and $u_2\in W^{1,q}(B)$ such that
\[
u=u_1+u_2 \qquad \mbox{ so that by linearity }\qquad T_j(u)=T_j(u_1)+T_j(u_2).
\]
From the definition of the functional $K$, we have 
\[
\begin{split}
t^{-\beta}\, K(t,T_j(u),\mathcal{D}^{-1,q}(B),L^q(B))&\le t^{-\beta}\,\Big(\|T_j(u_1)\|_{\mathcal{D}^{-1,q}(B)}+t\,\|T_j(u_2)\|_{L^q(B)}\Big)\\
&\le t^{-\beta}\,\Big(\|u_1\|_{L^q(B)}+t\,\|u_2\|_{W^{1,q}(B)}\Big)\\
\end{split}
\]
where we used \eqref{C1} and \eqref{C2}. By taking the infimum over the admissible pairs $(u_1,u_2)$, we thus get
\[
t^{-\beta}\, K(t,T_j(u),\mathcal{D}^{-1,q}(B),L^q(B))\le t^{-\beta}\, K(t,u,L^q(B),W^{1,q}(B)).
\]
By Lemma \ref{lm:1} again, the right-hand side belongs to $L^q((0,+\infty);1/t)$. Thus, the same property is true for the left-hand side and this concludes the proof.
\end{proof}
\subsection{Computation of an interpolation space}
We now compute the interpolation space occurring in Lemma \ref{lm:3}. In what follows we denote by $X^*$ the topological dual of a Banach space $X$.
\begin{lemma}[Interpolation of dual spaces]
\label{lm:4}
Let $0<\beta<1$ and $1<q<\infty$, then we have the following chain of identities
\[
\begin{split}
\left(\mathcal{D}^{-1,q}(B),L^{q}(B)\right)_{\beta,q}&=\left(\left(\mathcal{D}^{1,q'}_0(B),L^{q'}(B)\right)_{\beta,q'}\right)^*\\
&=\left(\left(L^{q'}(B),\mathcal{D}^{1,q'}_0(B)\right)_{1-\beta,q'}\right)^*\\
&=\left(\mathcal{D}^{1-\beta,q'}_0(B)\right)^*\\
&=\mathcal{D}^{\beta-1,q}(B),
\end{split}
\]
as Banach spaces.
\end{lemma}
\begin{proof}
The first identity is a consequence of the so-called {\it Duality Theorem} in real interpolation theory, see \cite[Theorem 3.7.1]{BL}. This result requires the space
\[
L^{q'}(B)\cap \mathcal{D}^{1,q'}_0(B)=\mathcal{D}^{1,q'}_0(B),
\]
to be dense both in $L^{q'}(B)$ and $\mathcal{D}^{1,q'}_0(B)$, an hypothesis which is of course verified.  
\par
The second identity is a basic fact in real interpolation, see \cite[Theorem 3.4.1]{BL}. On the other hand, the fourth identity is just the definition of dual space. 
\par
In order to conclude, we need to show the third identity, i.\,e.
\[
\left(L^{q'}(B),\mathcal{D}^{1,q'}_0(B)\right)_{1-\beta,q'}=\mathcal{D}^{1-\beta,q'}_0(B).
\]
Let us set for notational simplicity
\[
\mathcal{X}(B):=\left(L^{q'}(B),\mathcal{D}^{1,q'}_0(B)\right)_{1-\beta,q'}.
\]
For every $u\in L^{q'}(B)+\mathcal{D}^{1,q'}_0(B)$ and $t>0$, we also use the simplified notation
\[
K(t,u)=\inf_{u_1\in L^{q'}(B),\, u_2\in \mathcal{D}^{1,q'}_0(B)}\Big\{\|u_1\|_{L^{q'}(B)}+t\,\|u_2\|_{\mathcal{D}^{1,q'}_0(B)}\,: \, u=u_1+u_2\Big\}.
\]
For every $u\in\mathcal{X}\setminus\{0\}$ and $h\in\mathbb{R}^N\setminus\{0\}$, there exist $u_1\in L^{q'}(B)$ and $u_2\in \mathcal{D}^{1,q'}_0(B)$ such that
\begin{equation}
\label{dai}
u=u_1+u_2\qquad \mbox{ and }\qquad \|u_1\|_{L^{q'}(B)}+|h|\,\|u_2\|_{\mathcal{D}^{1,q'}_0(B)}\le 2\,K(|h|,u).
\end{equation}
Both $u_1$ and $u_2$ are extended to $\mathbb{R}^N\setminus \overline B$ by $0$.
Thus for $h\not=0$ we get\footnote{In the second inequality, we use the classical fact
\[
\int_{\mathbb{R}^N} \frac{|u_2(x+h)-u_2(x)|^{q'}}{|h|^{q'}}\,dx\le C\,\int_{\mathbb{R}^N} |\nabla u_2|^{q'}\,dx.
\]}
\[
\begin{split}
\int_{\mathbb{R}^N} \frac{|u(x+h)-u(x)|^{q'}}{|h|^{N+(1-\beta)\,q'}}\,dx&\le C\,\int_{\mathbb{R}^N} \frac{|u_1(x+h)-u_1(x)|^{q'}}{|h|^{N+(1-\beta)\,q'}}\,dx\\
&+C\,\int_{\mathbb{R}^N} \frac{|u_2(x+h)-u_2(x)|^{q'}}{|h|^{N+(1-\beta)\,q'}}\,dx\\
&\le C\,|h|^{-N-(1-\beta)\,q'}\,\|u_1\|_{L^{q'}(B)}^{q'}\\
&+C\,|h|^{-N+\beta\,q'}\,\|\nabla u_2\|_{L^{q'}(B)}^{q'}\\
&\le C\, |h|^{-N-(1-\beta)\,q'}\,\Big(\|u_1\|_{L^{q'}(B)}+|h|\,\|u_2\|_{\mathcal{D}^{1,q'}_0(B)}\Big)^{q'}.
\end{split}
\]
By using \eqref{dai}, we obtain
\[
\int_{\mathbb{R}^N} \frac{|u(x+h)-u(x)|^{q'}}{|h|^{N+(1-\beta)\,q'}}\,dx\le C\, |h|^{-N-(1-\beta)\,q'}\,K(|h|,u)^{q'}.
\]
We now integrate in $h$ over $\mathbb{R}^N$. This yields
\[
\begin{split}
\int_{\mathbb{R}^N}\int_{\mathbb{R}^N} \frac{|u(x+h)-u(x)|^{q'}}{|h|^{N+(1-\beta)\,q'}}\,dx\,dh&\le C\, \int_{\mathbb{R}^N} |h|^{-N-(1-\beta)\,q'}\,K(|h|,u)^{q'}\,dh\\
&=C\,N\,\omega_N\, \int_0^{+\infty} t^{(\beta-1)\,q'}\,K(t,u)^{q'}\,\frac{dt}{t}=C\, \|u\|^{q'}_{\mathcal{X}(B)},
\end{split}
\]
which shows that we have the continuous embedding
\[
\mathcal{X}(B)\hookrightarrow \mathcal{D}^{1-\beta,q'}_0(B).
\]
\vskip.2cm
We now need to show the converse embedding. Let us indicate by $R>0$ the radius of $B$, we further assume for simplicity that $B$ is centered at the origin. Let $u\in \mathcal{D}^{1-\beta,q'}_0(B)$, then in particular $u\in L^{q'}(B)$ (see Remark \ref{oss:poincare}) and thus from the writing
\[
u=u+0,
\]
by definition of $K$ we obtain
\begin{equation}
\label{primo}
\begin{split}
\int_{R/2}^{+\infty} t^{(\beta-1)\,q'}\, K(t,u)^{q'}\,\frac{dt}{t}&\le \left(\int_{R/2}^{+\infty} t^{(\beta-1)\,q'-1}\,dt\right)\, \|u\|_{L^{q'}(B)}^{q'}\\
&=\frac{R^{(\beta-1)\,q'}}{2^{(\beta-1)\,q'}\,(1-\beta)\,q'}\, \|u\|_{L^{q'}(B)}^{q'}\le C\, \|u\|_{\mathcal{D}_0^{1-\beta,q'}(B)}^{q'},
\end{split}
\end{equation}
for some $C=C(N,\beta,q)>0$. In the last passage we used Poincar\'e inequality for $\mathcal{D}_0^{1-\beta,q'}(B)$.
\par  
We need to deal with $0<t<R/2$. For this, we take a positive function $\psi\in C^{\infty}_0$ with support contained in $\{x\in\mathbb{R}^N\, :\, |x|<1\}$ and such that $\int_{\mathbb{R}^N} \psi\,dx=1$. Then we define
\[
\psi_t(x)=\frac{1}{t^N}\,\psi\left(\frac{x}{t}\right),\qquad x\in\mathbb{R}^N,\ t>0,
\]
and
\[
v_t(x)=u\left(\frac{R}{R-t}\,x\right),\qquad x\in\mathbb{R}^N,\ 0<t<\frac{R}{2},
\]
where we intend that $u$ is extended by $0$ to the whole $\mathbb{R}^N$. In particular, $v_t$ vanishes outside $B_{R-t}$.
Finally, for $0<t<R/2$ we set
\[
u_t(x)=v_t\ast \psi_t(x)=\int_{B_{R-t}} u\left(\frac{R}{R-t}\,y\right)\,\frac{1}{t^N}\,\psi\left(\frac{x-y}{t}\right)\,dy.
\]
By construction, for every $0<t<R/2$ the function $v_t$ is such that
\[
\int_{\mathbb{R}^N} |\nabla u_t|^{q'}\,dx<+\infty\qquad \mbox{ and }\qquad u_t\equiv 0\ \mbox{ in }\mathbb{R}^N\setminus B,
\]
thus $u_t\in \mathcal{D}^{1,q'}_0(B)$.
We observe that by Jensen inequality
\[
\|u-u_t\|_{L^{q'}(B)}^{q'}\le \int_{B}\int_{B_{R-t}} \left|u(x)-u\left(\frac{R}{R-t}\,y\right)\right|^{q'}\,\frac{1}{t^N}\,\psi\left(\frac{x-y}{t}\right)\,dy\,dx.
\]
Thus by using a change of variable and Fubini Theorem we get
\[
\begin{split}
\int_0^{R/2} &t^{(\beta-1)\,q'}\,\|u-u_t\|_{L^{q'}(B)}^{q'}\,\frac{dt}{t}\\
&\le \int_0^{R/2} \int_{B}\int_{B_{R-t}} t^{(\beta-1)\,q'}\,\left|u(x)-u\left(\frac{R}{R-t}\,y\right)\right|^{q'}\,\frac{1}{t^N}\,\psi\left(\frac{x-y}{t}\right)\,dy\,dx\,\frac{dt}{t}\\
&=\int_0^{R/2} \int_{B}\int_{B} \left(\frac{R-t}{R}\right)^N\,t^{(\beta-1)\,q'}\,\left|u(x)-u(z)\right|^{q'}\,\frac{1}{t^N}\,\psi\left(\frac{x}{t}-\frac{R-t}{R\,t}z\right)\,dz\,dx\,\frac{dt}{t}\\
&\le \int_{B}\int_{B}\left|u(x)-u(z)\right|^{q'}\left(\int_0^{R/2} \,t^{(\beta-1)\,q'-N}\,\psi\left(\frac{x-z}{t}+\frac{z}{R}\right)\,\frac{dt}{t}\right)\,dz\,dx.
\end{split}
\]
We now observe that 
\[
\left|\frac{x-z}{t}+\frac{z}{R}\right|>1\qquad \Longrightarrow \qquad \psi\left(\frac{x-z}{t}+\frac{z}{R}\right)=0,
\]
thus in particular
\[
\left|\frac{x-z}{t}\right|>1+\left|\frac{z}{R}\right|\qquad \Longrightarrow \qquad \psi\left(\frac{x-z}{t}+\frac{z}{R}\right)=0.
\]
Finally, for $x,z\in B$ we get
\[
\begin{split}
\int_0^{R/2} \,t^{(\beta-1)\,q'-N}\,\psi\left(\frac{x-z}{t}+\frac{z}{R}\right)\,\frac{dt}{t}&\le \int_0^{+\infty} \,t^{(\beta-1)\,q'-N}\,\psi\left(\frac{x-z}{t}+\frac{z}{R}\right)\,\frac{dt}{t}\\
&=\int_{\frac{|x-z|}{1+\frac{|z|}{R}}}^{+\infty} \,t^{(\beta-1)\,q'-N}\,\psi\left(\frac{x-z}{t}+\frac{z}{R}\right)\,\frac{dt}{t}\\
&\le \int_{\frac{|x-z|}{2}}^{+\infty} \,t^{(\beta-1)\,q'-N}\,\psi\left(\frac{x-z}{t}+\frac{z}{R}\right)\,\frac{dt}{t}\\
&\le C\,\|\psi\|_{L^\infty(\mathbb{R}^N)}\,|x-z|^{-N-(1-\beta)\,q'}.
\end{split}
\]
Thus, we obtain
\begin{equation}
\label{secondo}
\begin{split}
\int_0^{R/2} t^{(\beta-1)\,q'}\,\|u-u_t\|_{L^{q'}(B)}^{q'}\,\frac{dt}{t}&\le C\,\|\psi\|_{L^\infty(\mathbb{R}^N)}\,\int_B \int_B \frac{|u(x)-u(z)|^{q'}}{|x-z|^{N+(1-\beta)\,q'}}\,dx\,dz\\
&\le C'\, \|u\|^{q'}_{\mathcal{D}^{1-\beta,q'}_0(B)}.
\end{split}
\end{equation}
In order to finish the proof, we just need to show that 
\begin{equation}
\label{terzo}
\int_0^{R/2} t^{(\beta-1)\,q'}\,t^{q'}\,\|u_t\|_{\mathcal{D}^{1,q'}_0(B)}^{q'}\,\frac{dt}{t}\le C\, \|u\|^{q'}_{\mathcal{D}^{1-\beta,q'}_0(B)}.
\end{equation}
We first observe that (recall that $u\equiv 0$ outside $B$)
\[
\nabla u_t(x)=\int_{B_{R-t}} u\left(\frac{R}{R-t}\,y\right)\,\frac{1}{t^{N+1}}\,\nabla \psi\left(\frac{x-y}{t}\right)\,dy=\int_{\mathbb{R}^N} u\left(\frac{R}{R-t}\,y\right)\,\frac{1}{t^{N+1}}\,\nabla \psi\left(\frac{x-y}{t}\right)\,dy,
\]
and by the Divergence Theorem
\[
\int_{\mathbb{R}^N}\frac{1}{t^{N+1}}\,\nabla \psi\left(\frac{x-y}{t}\right)\,dy=0.
\]
Thus we obtain as well
\[
-\nabla u_t(x)=\int_{\mathbb{R}^N} \left[u\left(\frac{R}{R-t}\,x\right)-u\left(\frac{R}{R-t}\,y\right)\right]\,\frac{1}{t^{N+1}}\,\nabla \psi\left(\frac{x-y}{t}\right)\,dy,
\]
and by H\"older inequality
\[
\begin{split}
\|u_t\|_{\mathcal{D}^{1,q'}_0(B)}^{q'}&=\int_{\mathbb{R}^N} |\nabla u_t|^{q'}\,dx\\
& \le \int_{\mathbb{R}^N} \left(\int_{\mathbb{R}^N} \left|u\left(\frac{R}{R-t}\,x\right)-u\left(\frac{R}{R-t}\,y\right)\right|^{q'}\,\frac{1}{t^{N+1}}\,\left|\nabla \psi\left(\frac{x-y}{t}\right)\right|\,dy\right)\\
&\times\left(\int_{\mathbb{R}^N}\frac{1}{t^{N+1}}\,\left|\nabla \psi\left(\frac{x-y}{t}\right)\right|\,dy\right)^{q'-1}\,dx\\
&= \frac{\|\nabla \psi\|^{q'-1}_{L^1(\mathbb{R})}}{t^{q'-1}}\,\int_{\mathbb{R}^N} \int_{\mathbb{R}^N} \left|u\left(\frac{R}{R-t}\,x\right)-u\left(\frac{R}{R-t}\,y\right)\right|^{q'}\,\,\frac{1}{t^{N+1}}\,\left|\nabla \psi\left(\frac{x-y}{t}\right)\right|\,dy\,dx\\
&\le \frac{\|\nabla \psi\|^{q'-1}_{L^1(\mathbb{R})}}{t^{q'-1}}\,\int_{\mathbb{R}^N} \int_{\mathbb{R}^N} \left|u\left(z\right)-u\left(w\right)\right|^{q'}\,\,\frac{1}{t^{N+1}}\,\left|\nabla \psi\left(\frac{R-t}{R\,t}\,(z-w)\right)\right|\,dz\,dw.
\end{split}
\]
This yields
\begin{equation}
\label{quasi6}
\begin{split}
\int_0^{R/2} &t^{(\beta-1)\,q'}\,t^{q'}\,\|u_t\|_{\mathcal{D}^{1,q'}_0(B)}^{q'}\,\frac{dt}{t}\\
&\le C\,\int_0^{R/2} t^{(\beta-1)\,q'}\,\int_{\mathbb{R}^N} \int_{\mathbb{R}^N} \left|u\left(z\right)-u\left(w\right)\right|^{q'}\,\,\frac{1}{t^{N+1}}\,\left|\nabla \psi\left(\frac{R-t}{R\,t}\,(z-w)\right)\right|\,dz\,dw\,dt\\
&=C\,\int_{\mathbb{R}^N} \int_{\mathbb{R}^N} |u(z)-u(w)|^{q'}\,\left(\int_0^{R/2} t^{(\beta-1)\,q'}\,\frac{1}{t^{N}}\,\left|\nabla \psi\left(\frac{R-t}{R\,t}\,(z-w)\right)\right|\,\frac{dt}{t}\right)\,dz\,dw\\
&\le C\,\int_{\mathbb{R}^N} \int_{\mathbb{R}^N} |u(z)-u(w)|^{q'}\,\left(\int_0^{+
\infty} t^{(\beta-1)\,q'}\,\frac{1}{t^{N}}\,\left|\nabla \psi\left(\frac{R-t}{R\,t}\,(z-w)\right)\right|\,\frac{dt}{t}\right)\,dz\,dw.
\end{split}
\end{equation}
We now observe that 
\[
\frac{R-t}{R}\,\frac{|z-w|}{t}>1\qquad \Longrightarrow\qquad \nabla \psi\left(\frac{R-t}{R\,t}\,(z-w)\right)=0,
\]
thus in particular for $0<t<R/2$ we have
\[
\frac{1}{2}\,\frac{|z-w|}{t}>1\qquad \Longrightarrow\qquad \nabla \psi\left(\frac{R-t}{R\,t}\,(z-w)\right)=0.
\]
This implies that for $z,w\in\mathbb{R}^N$ we have
\[
\begin{split}
\int_0^{\infty} t^{(\beta-1)\,q'}\,\frac{1}{t^{N}}\,\left|\nabla \psi\left(\frac{R-t}{R\,t}\,(z-w)\right)\right|\,\frac{dt}{t}&=\int_{\frac{|z-w|}{2}}^{+\infty} t^{(\beta-1)\,q'}\,\frac{1}{t^{N}}\,\left|\nabla \psi\left(\frac{R-t}{R\,t}\,(z-w)\right)\right|\,\frac{dt}{t}\\
%&=\int_{\frac{|z-w|}{2}}^{+\infty} t^{(\alpha-1)\,q'}\,\frac{1}{t^{N}}\,\left|\nabla \psi\left(\frac{R-t}{R\,t}\,(z-w)\right)\right|\,\frac{dt}{t}\\
&\le C\,\|\nabla \psi\|_{L^\infty(\mathbb{R}^N)}\,|z-w|^{-N-(1-\beta)\,q'}.
\end{split}
\]
%On the other hand, if $|z-w|>R$ then we get
%\[
%\left|\frac{R-t}{R\,t}\,(z-w)\right|>\frac{R-t}{t}>1,\qquad \mbox{ for every } 0<t<\frac{R}{2},
%\]
%thus in this case
%\[
%\int_0^{R/2} t^{(\beta-1)\,q'}\,\frac{1}{t^{N}}\,\left|\nabla \psi\left(\frac{R-t}{R\,t}\,(z-w)\right)\right|\,\frac{dt}{t}=0.
%\]
Then from \eqref{quasi6} we obtain \eqref{terzo}. 
We can now conclude the proof of the embedding
\[
\mathcal{D}^{1-\beta,q'}_0(B)\hookrightarrow \mathcal{X}(B).
\]
From \eqref{primo}, \eqref{secondo} and \eqref{terzo} we get
\[
\begin{split}
\int_{0}^{+\infty} t^{(\beta-1)\,q'}\, K(t,u)^{q'}\,\frac{dt}{t}&=\int_0^{R/2} t^{(\beta-1)\,q'}\, K(t,u)^{q'}\,\frac{dt}{t}+\int_{R/2}^{+\infty} t^{(\beta-1)\,q'}\, K(t,u)^{q'}\,\frac{dt}{t}\\
&\le C\,\int_0^{R/2} t^{(\beta-1)\,q'}\, \|u-u_t\|^{q'}\,\frac{dt}{t}\\
&+C\,\int_0^{R/2} t^{(\beta-1)\,q'}\,t^{q'}\,\|u_t\|_{\mathcal{D}^{1,q'}_0(B)}^{q'}\,\frac{dt}{t}\\
&+\int_{R/2}^{+\infty} t^{(\beta-1)\,q'}\, K(t,u)^{q'}\,\frac{dt}{t} \le C\,\|u\|^{q'}_{\mathcal{D}^{1-\beta,q'}_0(B)},
\end{split}
\]
for some $C=C(N,\beta,q)>0$. This concludes the proof.
\end{proof}

\end{document}